\documentclass[a4paper,12pt]{article}
\usepackage{amsmath,amsfonts,amssymb,units,amsthm,mathrsfs}

\usepackage{geometry}
\newtheorem{theorem}{Theorem}[section]

\newtheorem{corollary}[theorem]{Corollary}
\newtheorem{lemma}[theorem]{Lemma}

\newtheorem{proposition}[theorem]{Proposition}

\newcommand{\ud }{\,{\rm d}}
\newcommand{\rf}[1]{\eqref{#1}}

\newcommand{\bbfR}{{\mathbb R}}
\newcommand{\bbfC}{{\mathbb C}}

\newcommand{\ve}{{\varepsilon }}
\newcommand{\se}{{\text{\rm{e}}}}

\title{Asymptotic stability of Landau solutions to Navier-Stokes system\footnote{This work was partially supported 
by the MNiSzW  grants No.  N N201 365736  and N N201 418839,
 and 
the Foundation for Polish Science operated within the
Innovative Economy Operational Programme 2007-2013 funded by European
Regional Development Fund (Ph.D. Programme: Mathematical
Methods in Natural Sciences).}
}

\author{Grzegorz Karch \& Dominika Pilarczyk\\
\rule{0pt}{15pt}\\
\small Instytut Matematyczny, Uniwersytet Wroc\l awski,\\
\small pl. Grunwaldzki 2/4, 50-384 Wroc\l aw, Poland \\
\small e-mail: \small \texttt {\{grzegorz.karch,dominika.pilarczyk\}{@}math.uni.wroc.pl}}

\date{\rule{0pt}{15pt}\\ \today}

\begin{document}

\maketitle

\begin{abstract}
\noindent
It is known that the three dimensional Navier-Stokes system for an incompressible fluid in the whole space has a one parameter family of explicit stationary solutions, which are axisymmetric and homogeneous of degree $-1$. We show that these solutions are asymptotically stable under any $L^2$-perturbation.
\end{abstract}
\section{Introduction}
\setcounter{equation}{0}

The initial value problem for the Navier--Stokes system describing a motion of a viscous incompressible fluid in the whole three dimensional space has the form
\begin{align}
\label{NS eq}
         u_t-\Delta u+(u\cdot \nabla )u+\nabla p&=F,\quad  (x,t) \in \bbfR^3\times (0,\infty )\\
\label{NS div}         
       {\rm div }\ u&=0, \\
\label{NS in d}       
       u(x,0)&=u_0 (x).
\end{align}  
Here, the velocity $u=(u_1,u_2,u_3)$ and the scalar pressure $p$ are unknown. Moreover, $u_0$ and $F$ denote a given initial velocity and a given external force, respectively. 

It is well-known, since the pioneer work of Leray \cite{L}, that for each $u_0 \in \big(L^2(\bbfR^3)\big)^3$ satisfying $\text{\rm div}\,u_0=0$ and for $F\equiv 0$,  problem \rf{NS eq} possesses a weak solution, satisfying a suitable energy inequality (see the monograph \cite{T} for analogous results with nonzero $F$). The uniqueness and the regularity of weak solutions still remain open. In \cite{L}, Leray posed a question whether a weak solution $u=u(x,t)$ tends to zero in $L^2(\bbfR^3)$ as $t\rightarrow \infty$, which was affirmatively solved by Kato \cite{K} in the case of strong solutions and Masuda \cite{Mas} for weak solutions satisfying a strong energy inequality. Next, Schonbek \cite{S2} obtained decay rates for the $L^2$-norm of weak solutions using elementary properties of the Fourier transform. The ideas from \cite{S2} were developed and generalized by Wiegner \cite{W}. We refer the reader to monographs \cite{T,Lem} for results on the existence of weak and strong solutions to \rf{NS eq}--\rf{NS in d} and to the review article \cite{C} for a discussion of recent results of the large time behavior of solutions. 

If $F\equiv 0$ in problem \rf{NS eq}--\rf{NS in d}, the $L^2$-decay of weak solutions can be understood as the global asymptotic stability in $L^ 2(\bbfR^3)$ of the trivial stationary solution $(u,p)=(0,0)$. In this work, we address analogous questions on the global asymptotic stability of the family of stationary solutions to \rf{NS eq}--\rf{NS div} given by the following explicit formulas
\begin{align}\label{L sol}
   v_ c^1(x)&=2\frac{c|x|^2-2x_1|x|+cx_1^2}{|x|(c|x|-x_1)^2},  &v_c^2(x)&=2\frac{x_2(cx_1-|x|)}{|x|(c|x|-x_1)^2},
\nonumber\\
   v_c^3(x)&=2\frac{x_3(cx_1-|x|)}{|x|(c|x|-x_1)^2},   &p_c(x)&=4\frac{cx_1-|x|}{|x|(c|x|-x_1)^2},
\end{align}
where $|x|=\sqrt{x_1^2+x_2^2+x_3^2}$ and $c$ is an arbitrary constant such that $|c|>1$. The functions $v_c $ and $p_c$ defined in \rf{L sol} satisfy \rf{NS eq} with $F\equiv 0$ in the pointwise sense for every $x\in \bbfR^3 \setminus \{0\}$. 
On the other hand, if one treats them as a distributional or generalized solution to \rf{NS eq} in the whole $\bbfR^3$, they correspond to the very singular external force $F=(b(c)\delta _0, 0,0)$, where the parameter $b\not =0$ depends on $c$ and $\delta _0$ stands for the Dirac measure. 
Indeed, in \cite[Proposition 2.1.]{CK} (see also \cite[p. 206]{B}), it was shown that for every test function $\varphi \in C_c^\infty (\bbfR^3)$ the following equalities hold true
\begin{align*}
   \int_{\bbfR^3}v_c(x)\cdot \nabla \varphi (x)\ud x &=0 \\
\intertext{and}
   \int_{\bbfR^3} \Big(\nabla v_c^k\cdot \nabla \varphi -v_c^kv_c\cdot \nabla \varphi -p_c \frac{\partial }{\partial x_k} \varphi \Big)\ud x&=
   \begin{cases}
   b(c)\varphi(0) & \text{if } \ k=1,\\
   0& \text{if } \ k=2,3,
   \end{cases}
\end{align*}
where 
\begin{equation}\label{bc}
   b(c)=\frac{8\pi c}{3(c^2-1)}\Big( 2+6c^2 -3c(c^2-1)\log\Big(\frac{c+1}{c-1}\Big) \Big).
\end{equation}
In particular, the function $b=b(c)$ is decreasing on $(-\infty , -1)$ and $(1, +\infty )$. Moreover, $\lim_{c\rightarrow  1}b(c)=+\infty $, $\lim_{c\rightarrow  -1}b(c)=-\infty $ and $\lim_{|c|\rightarrow  \infty }b(c)=0$.

These explicit stationary solutions to \rf{NS eq}--\rf{NS div} were first calculated by Landau \cite{Landau} and now they can be found in standard textbooks (see {\it e.g.} \cite[p. 82]{LL} and \cite[p. 206]{B}). Let us also recall that the stationary solutions \rf{L sol} were also independently found by Squire \cite{Sq} and discussed in \cite{CK,TX} from a slightly different point of view. The main idea of Landau's calculation is that if we impose the additional axi-symmetry requirement, the stationary Navier-Stokes system
\begin{equation}\label{stat}
   -\Delta u+(u\cdot \nabla )u+\nabla p=0, \qquad {\rm div}\  u=0,
\end{equation}
 reduces to a system of ODEs which can be solved explicitly in terms of elementary functions. Moreover, \v{S}ver\'ak \cite{Sv} proved recently that even if we drop the requirement of axi-symmetry, then the Landau solutions \rf{L sol} are still the only solutions of \rf{stat} which are invariant under the natural scaling. More precisely, he proved that if $u:\bbfR^3 \setminus \{0\} \rightarrow  \bbfR^3$ is a non-trivial smooth solution of  \rf{stat} satisfying $\lambda u(\lambda x) = u(x) $ for all $x\in \bbfR^3\setminus\{0\}$ and each $\lambda > 0$, then $(u,p)=(v_c,p_c)$ is given by formulas \rf{L sol} (modulo a rotation of $\bbfR^3$).

The goal of this work is to show that problem \rf{NS eq}--\rf{NS in d} has a weak solution for every initial datum of the form $u_0=v_c+w_0$, where $w_0\in L^2(\bbfR^3)$ and the external force $F=(b(c)\delta_0,0,0)$ with $b(c)$ defined in \rf{bc}, provided $|c|$ is sufficiently large. Moreover, this solution converges, as $t\rightarrow \infty$, towards the stationary solution \rf{L sol}. In other words, we show that the flow described by the Landau solution is, in some sense, asymptotically stable under any $L^2$-perturbation.

The existence and stability of stationary solutions corresponding to nontrivial external forces are well understood in the case of bounded domains, see for example \cite{DG}. For related results in exterior domains, we refer the reader to \cite{F3,H} and to the references therein. The existence and the stability of stationary solutions in $L^p$ with $p \geqslant n$, where $n$ is the dimension of the space, is obtained in \cite{S}, under the condition that the Reynolds number is sufficiently small, and in \cite{CK, CK2, KY1, KY2,Y} under the assumption that the external force is sufficiently small. The stability of small stationary solutions of \rf{NS eq}--\rf{NS in d} in $L^p(\bbfR^3)$ with $p<3$ has been studied recently in \cite{BS, BBIS}.

\textbf{Notation.} In this work, the usual norm of the Lebesgue space $L^p (\bbfR^3)$ is denoted by $\|\cdot\|_p$ for any $p \in [1,\infty]$. $C^\infty_c (\bbfR^3)$ denotes the set of smooth and compactly supported functions. Here, we work with the Sobolev space $H^1(\bbfR^3 )=\{ f\in L^2(\bbfR^3 ): \nabla f \in L^2(\bbfR^3 )\}$ and with its homogeneous counterpart $\dot H^1(\bbfR^3 )=\{ f\in L^1_{loc} (\bbfR^3 ): \nabla f\in L^2(\bbfR^3 )\}$. We use the following notation for the Banach spaces of divergence free vector fields: $L^p_\sigma (\bbfR^3)  =\{ u\in \big( L^p(\bbfR^3)\big)^3 : \textrm{div}\ u=0 \}$ and $\dot H^1_\sigma (\bbfR^3) = \{ u\in \big( \dot H^1(\bbfR^3)\big)^3 : \textrm{div } u=0\}$ supplemented with usual norms. The constants (always independent of $x$ and $t $) will be denoted by the same letter $C$, even if they vary from line to line.

\section{Results and comments}
\setcounter{equation}{0}

We denote by $u=u(x,t)$ a solution of the Navier--Stokes system \rf{NS eq}--\rf{NS in d} with the external force $F=b(c)\delta _0$, where $b(c)$ is defined in \rf{bc}, and the initial datum $u_0=v_c +w_0$, where $v_c$ is the singular stationary solution \rf{L sol} and $w_0\in L^2_\sigma (\bbfR^3)$. Then the functions $w(x,t)=u(x,t)-v_c(x)$  and $\pi (x)=p(x)-p_c(x)$ satisfy the initial value problem
\begin{align}
   \label{w:eq} w_t-\Delta w+(w\cdot  \nabla )w+ (w\cdot  \nabla )v_c +(v_c \cdot \nabla )w+\nabla \pi&=0,\\
   \label{w:d}   {\rm div}\  w&=0,\\
   \label{w:ini} w(x,0)&=w_0(x)\ .
\end{align} 
The goal of this work is to show the existence of a global-in-time weak solution to problem \rf{w:eq}--\rf{w:ini} in a usual energy space (see \rf{XT} below) and to study its convergence in $L^2_\sigma (\bbfR^3)$ as $t\rightarrow \infty$ zero. As in the classical work by Leray \cite{L}, these solutions satisfy a suitable energy inequality. Here, however, in the proof of the $L^2$-decay of solutions to \rf{w:eq}--\rf{w:ini}, we need a strong energy inequality, introduced by Masuda \cite{Mas} for the Navier-Stokes system \rf{NS eq}--\rf{NS in d}.

In our analysis, the crucial role is played by the Hardy-type inequality
\begin{equation}\label{hardy:0}
   \bigg|\int_{\bbfR^3} w\cdot (w\cdot \nabla )v_c \ud x \bigg|\leqslant K(c)\|\nabla \otimes w\|_2^2 ,
\end{equation}
which is valid for all $w\in \dot H^1(\bbfR^3)$. Here, the function $K=K(c)>0$ satisfies $\lim_{|c|\rightarrow 1}K(c)=+\infty$ and $\lim_{|c|\rightarrow +\infty}K(c)=0$ (see Theorem \ref{l:H}, below), hence, there exists $c_0>1$ such that
\begin{equation}\label{c0}
   K(c)<1 \quad \text{\rm for all  $\in \bbfR$ satisfying }\quad |c|\geqslant c_0>1 \, .
\end{equation} 
In the next section, we deduce inequality \rf{hardy:0} from the classical Hardy inequality 
\begin{equation}\label{hardy}
   \int_{\bbfR^3}\frac{|w(x)|^2}{|x|^2} \ud x \leqslant 4 \int_{\bbfR^3}|\nabla w(x)|^2 \ud x \ \ \textrm{for all } w\in \dot H(\bbfR^3)
\end{equation}
which proof can be found {\it e.g.} in \cite[Ch. I. 6]{L}.
 
First, we state the counterpart of the Leray result on the existence of weak solutions to the initial value problem \rf{w:eq}--\rf{w:d}.

\begin{theorem}\label{th:ex}
Assume that $c_0>1$ satisfies \rf{c0}. For each $c\in \bbfR$ such that $|c|>c_0$, every $w_0 \in L^2_\sigma (\bbfR^3 )$, and every $T>0$ 
problem \rf{w:eq}--\rf{w:ini} has a weak solution in the energy space
\begin{equation}\label{XT}
  X_T=L^{\infty }_w \big( [0,T], L^2_\sigma (\bbfR^3)\big)\cap L^2\big([0,T], \dot H_\sigma^1(\bbfR^3)\big) ,
\end{equation}
which satisfies the strong energy inequality
\begin{equation}\label{en:inq}
   \|w(t)\|_2^2+ 2(1-K(c))\int_s^t \|\nabla \otimes w(\tau )\|_2^2 \ud \tau \leqslant \|w (s)\|_2^2\ 
\end{equation}
for almost all $s\geqslant 0$, including $s=0$ and all $t\geqslant s$.
\end{theorem}

Recall that, following a classical approach, a function $w\in X_T$ is a weak solution of problem \rf{w:eq}--\rf{w:ini}
if 
\begin{equation}\label{weak}
\begin{split}
\big(  w(t), \varphi(t)  \big)
+&\int_s^t \Big[ 
\big( \nabla w, \nabla \varphi\big)
+ \big( w\cdot  \nabla w, \varphi\big)
+ \big(w\cdot  \nabla v_c , \varphi\big)
+\big(v_c \cdot \nabla w, \varphi\big) 
 \Big]\,d\tau \\
 &= \big(w(s), \varphi(s)\big) + \int_s^t \big(w, \varphi_\tau \big)\, d\tau
 \end{split}
\end{equation}
for all $t\geq s\geq 0$ and all $\varphi \in C([0, \infty), H_\sigma^1(\bbfR^3))\cap C^1([0,\infty), L^2_\sigma (\bbfR^3))$, where
$(\cdot, \cdot)$ is the standard $L^2$-inner product.   Notice that each term 
in \eqref{weak}  containg the singular function $v_c$ 
is convergent due to the Hardy inequality  \eqref{hardy}, see calculations in \eqref{vc:est:1}-\eqref{vc:est:2}, below.

The proof of Theorem \ref{th:ex} follows the well-known argument which we recall in Section~\ref{existence}. Here, we only recall that the most general result on the existence of weak solutions to the Navier-Stokes system in the exterior domain satisfying the strong energy inequality was proved by Miyakawa and Sohr \cite{MS}.

The decay in $L^2(\bbfR^3)$ of weak solutions from Theorem \ref{th:ex} is the main result of this work.
\begin{theorem}\label{lim w:th} 
Every weak solution $w=w(x,t)$ to problem \rf{w:eq}--\rf{w:ini} satisfying the strong energy inequality \rf{en:inq} has the property: $\lim_{t\rightarrow \infty } \|w(t)\|_2=0$ .
\end{theorem}

Under additional assumptions on initial data, we find also the decay rate of $\|w(t)\|_2$.
\begin{corollary}\label{cor. decay}
Under the assumptions of Theorem \ref{lim w:th} if, moreover, $w_0 \in L^p(\bbfR^3) \cap L_\sigma^2(\bbfR^3)$ for some $p\in (\frac{6}{5},2)$, then there exists $C>0$ such that 
\begin{equation}\label{decay}
    \|w(t)\|_2 \leqslant Ct^{-\frac{3}{2}(\frac{1}{p}-\frac{1}{2})}
\end{equation}
for all $t>0$.
\end{corollary}
\section{Hardy-type inequality and existence of weak solutions}\label{existence}
\setcounter{equation}{0}

First, we prove elementary pointwise estimates of the components of the matrix $\nabla v_c$.

\begin{lemma}\label{l:K jk}
Let $|c|>1$. There exist functions $K_{j,k}: (-\infty ,-1)\cup (1, \infty )\rightarrow (0, \infty )$ for every $j, k \in \{1,2,3\}$ such that for all $x\in \bbfR^3 \setminus \{0\}$, we have
\begin{equation}\label{K jk}
   \bigg|\partial_{x_j}v_c ^k(x)\bigg|\leqslant \frac{K_{j,k}(c)}{|x|^2} .
\end{equation}
Moreover, functions $K_{j,k}=K_{j,k}(c)$ have the following properties: $\lim_{|c|\rightarrow 1}K_{j,k}(c)=+\infty$ and $\lim_{|c|\rightarrow +\infty}K_{j,k}(c)=0$ for all $j,k\in \{1,2,3\}$.
\end{lemma}
 
\begin{proof}
It follows from the explicit formula for $v_c $ and $p_c$ (\textit{cf.} \rf{L sol}) that
\begin{equation}\label{v:p}
   v_c^1(x) = \frac{1}{2}p(x)x_1 +\frac{2}{c|x|-x_1},\quad
   v_c^2(x) = \frac{1}{2}p(x)x_2 ,\quad
   v_c^3(x) = \frac{1}{2}p(x)x_3 ,
\end{equation}
and
\begin{equation}\label{nabla p}
  \nabla p(x)=\frac{4}{|x|^3(c|x|-x_1)^3}\begin{pmatrix}(c^2-2)|x|^3+3c|x|^2x_1-3c^2|x|x_1^2+cx_1^3\\cx_2(2|x|^2-3c|x|x_1+x_1^2)\\cx_3(2|x|^2-3c|x|x_1+x_1^2)\end{pmatrix}.
\end{equation}
Moreover, using the expression for $p_c$ from \rf{L sol} and the notation $s=\nicefrac{x_1}{|x|}$, we obtain
\begin{equation*}
   |p_c(x)|\leqslant \frac{4}{|x|^2}\bigg|\sup_{s\in [-1,1]}\frac{cs-1}{(c-s)^2}\bigg|=k_p(c)\frac{1}{|x|^2},
\end{equation*}
where $k_p(c)=\frac{4}{|c|-1}$. In the same way by \rf{nabla p}, we have
\begin{align*}
  |x_i\partial_{x_1}p_c(x)|&\leqslant \frac{4}{|x|^2}\bigg|\sup_{s\in [-1,1]}\frac{cs^3-3c^2s^2+3cs+c^2-2}{(c-s)^3}\bigg|=k_{i,1}(c)\frac{1}{|x|^2}\\
\intertext{and}
  |x_i\partial_{x_2}p_c(x)|&\leqslant \frac{4c}{|x|^2}\bigg|\sup_{s\in [-1,1]}\frac{s^2-3cs+2}{(c-s)^3}\bigg| =k_{i,2}(c)\frac{1}{|x|^2},
\end{align*}
where $k_{i,1}=\frac{8}{1-|c|}$ and $k_{i,2}=\frac{12c}{(|c|-1)^2}$ for $i\in \{1,2,3\}$. Now, using the representation of $v_c$ in terms of $p_c$ from \rf{v:p}, we proceed in an analogous way to estimate all coefficients of the matrix $\{ \partial_{x_j}v_c^k(x)\}_{j,k=1}^3$. 
\end{proof}

The following theorem is the immediate consequence of Lemma \ref{l:K jk} and of the classical Hardy inequality \rf{hardy}.

\begin{theorem}[Hardy-type inequality]\label{l:H}
There exists a function $K:(-\infty ,-1)\cup (1, \infty )\rightarrow (0, \infty )$ with the following properties 
\begin{equation*}
\lim_{|c|\rightarrow 1}K(c)=+\infty \quad \textit{and }\quad \lim_{|c|\rightarrow +\infty}K(c)=0
\end{equation*}
such that for all vector fields  $w\in \dot H^1(\bbfR^3)$, we have $w\cdot (w\cdot \nabla )v_c \in L^1(\bbfR^3)$ together with the inequality
\begin{equation}\label{eq:H}
   \bigg|\int_{\bbfR^3} w\cdot (w\cdot \nabla )v_c \ud x \bigg|\leqslant K(c)\|\nabla \otimes w\|_2^2 .
\end{equation}
\end{theorem}

\begin{proof}
Applying Lemma \ref{l:K jk}, we get
\begin{align*}
  H(w)&\equiv \bigg|\int_{\bbfR^3} w\cdot (w\cdot \nabla )v_c \ud x \bigg| \leqslant \sum_{j,k=1}^3\int_{\bbfR^3} |w_jw_k| |\partial_{x_j}v_c ^k|  \ud x\leqslant \int_{\bbfR^3}\frac{\tilde K(c)}{|x|^2} \sum_{j,k=1}^3|w_j| |w_k| \ud x,
\end{align*}
where $\tilde K(c)=\max_{j,k \in\{1,2,3\}}K_{j,k}(c)$. Using the elementary inequality $a\cdot b\leqslant (a^2+b^2)/2$, we obtain
\begin{equation*}
   H(w)\leqslant \frac{1}{2}\int_{\bbfR^3}\frac{\tilde K(c)}{|x|^2}\bigg( \sum_{j,k=1}^3|w_j|^2+ \sum_{j,k=1}^3|w_k|^2\bigg)  \ud x=3\tilde K(c)\int_{\bbfR^3}\frac{|w|^2}{|x|^2} \ud x .
\end{equation*}
Finally, from the classical Hardy inequality \rf{hardy}, we have $H(w)\leqslant  K(c)\|\nabla \otimes w\|_2^2$, where $K(c)=12 \tilde K(c)$, which completes the proof of Theorem \ref{l:H}.
\end{proof}

Now, we are in a position to sketch the construction of weak solutions to problem \rf{w:eq}--\rf{w:ini}.

\begin{proof}[Proof of Theorem \ref{th:ex}]
This is the standard reasoning based on the Galerkin method. Since $H_\sigma^1(\bbfR^3)$ is separable, there exists a sequence $g_1, ..., g_m,...$ which is free and total in $H_\sigma^1(\bbfR^3)$. For each $m$, we define an  approximate solution $w_m=\sum_{i=1}^{m}d_{im}(t)g_i$, which satisfy the following system of ordinary differential equations
\begin{equation}\label{w m}
\begin{split} 
  \big(w'_m(t),g_j\big)&+\big(\nabla w_m (t), g_j\big)+\big((w_m(t)\cdot  \nabla )w_m(t),g_j\big)+ \big((w_m(t)\cdot  \nabla )v_c,g_j\big) \\
&+\big((v_c \cdot \nabla )w_m(t),g_j\big)+\big(\nabla \pi, g_j\big)=0 \qquad \text{for}\quad j=1,...,m,
\end{split}
\end{equation}
where $(f,g)=\int_{\bbfR^3}f(x) \cdot g(x) \ud x$. 

Let us prove that both terms in \rf{w m} containing the singular functions $\nabla v_c$ and $v_c$ are convergent. First, using the estimates from Lemma \ref{l:K jk} as in the proof of Theorem~\ref{l:H}, we obtain
\begin{equation}\label{vc:est:1}
\begin{split} 
 \big((w_m(t)\cdot  \nabla )v_c,g_j\big) &\leqslant \sum_{k,\ell=1}^3 \int_{\bbfR^3} |w_m^k g_j^\ell| |\partial_{x_k}v_c^\ell| \ud x \\
&\leqslant \frac{1}{2}\int_{\bbfR^3} \frac{\tilde{K}(c)}{|x|^2}\Big(\sum_{k,\ell =1}^3|w_m^k|^2+ |g_j^\ell |^2 \Big) \ud x .
\end{split}
\end{equation} 
Each term on the right-hand side of \eqref{vc:est:1} is finite due to the Hardy inequality \rf{hardy}. 
Next, using the explicit formulas \rf{L sol} we immediately obtain $ |\cdot |v_c \in \left(L^\infty(\bbfR^3)\right)^3, $ hence  the Schwarz inequality implies
\begin{equation} \label{vc:est:2}
    \big((v_c \cdot \nabla )w_m(t),g_j\big) \leqslant \big\| |\cdot |v_c \big\|_{\infty } \big\| |\cdot |^{-1} g_j \big\|_2 \|\nabla w_m \|_2.
\end{equation}
The right-hand side of this inequality is finite because 
 $ |\cdot |^{-1} g_j \in L^2(\bbfR^3)$ by the Hardy inequality  \rf{hardy}, again.

Now, we obtain \textit{a priori} estimate of the sequence $\{w_m\}_{m=1}^{\infty }$ by multiplying \rf{w m} by $d_{jm}$ and adding the resulting equations for $j=1,2,..., m$. Taking into account that ${\rm div }\ w_m =0$, we get 
\begin{equation*}
   \frac{1}{2}\frac{\ud}{\ud t}\|w_m(t)\|^2_2+\| \nabla \otimes w_m(t)\|_2^2+\big((w_m(t)\cdot  \nabla )v_c,w_m(t)\big)=0 .
\end{equation*}   
Consequently, using inequality \rf{eq:H} and integrating from $s$ to $t$, we obtain the estimate 
\begin{equation*}
    \|w_m(t)\|^2_2 +2 \big(1-K(c)\big) \int_s^t\| \nabla w_m(s)\|_2^2 \ud s\leqslant \|w (s)\|_2^2 .
\end{equation*}
Now, repeating the classical reasoning from {\it e.g.} \cite[Ch. III. Thm. 3.1]{T}, we obtain the existence of a weak solution in the energy space $X_T$ defined in \rf{XT}, which satisfies strong energy inequality \rf{en:inq}.
\end{proof}

\section{Linearized equation}\label{lin:eq}
\setcounter{equation}{0}

In the proof of the $L^2$-decay of weak solutions to problem \rf{w:eq}--\rf{w:ini}, we use properties of solutions to the linearized Cauchy problem 
\begin{align}
  \label{z:eq} z_t-\Delta z +(z \cdot \nabla )v_c +(v_c \cdot \nabla )z + \nabla \pi &=0, \quad (x,t)\in \bbfR^3 \times (0,\infty ) ,\\
  \label{z:d} \textrm{div}\ z&=0,\\
  \label{z:ini} z(x,0)&=z_0(x), \quad x\in \bbfR^3.
\end{align} 
Let us first recall that the Leray projector on divergence-free vector fields is defined by the formula $\mathbb{P} v =v - \nabla \Delta ^{-1}(\nabla \cdot v)$
for sufficiently smooth vectors $v=(v_1(x),v_2(x),v_3(x))$. To give a meaning to $\mathbb{P}$, it suffices to use the Riesz transforms $R_j$ which are the pseudo-differential operators defined in the Fourier variables as $\widehat {R_kf}(\xi )=\frac{i\xi }{|\xi|}\widehat{f}(\xi)$. Here, the Fourier transform of an integrable function $v$ is given by $\widehat{v}(\xi )=(2\pi )^{-\nicefrac{3}{2}}\int_{\bbfR^3}\se^{-ix\cdot \xi}v(x) \ud x$. Applying these well-known operators we define $(\mathbb{P}v)_j=v_j+\sum_{k=1}^3 R_jR_kv_k$.

Using the Leray projector $\mathbb{P}$, we can formally transform system \rf{z:eq}--\rf{z:d} into
\begin{align*}
   z_t-\Delta z+\mathbb{P}\Big((z \cdot \nabla )v_c\Big) +\mathbb{P}\Big((v_c \cdot \nabla )z \Big)&=0.
\end{align*}
Now, for simplicity, let us denote the linear operator
\begin{equation}\label{L:eq}
   \mathcal{L}z=-\Delta z+\mathbb{P}\Big((z \cdot \nabla )v_c\Big) +\mathbb{P}\Big((v_c \cdot \nabla )z\Big)
\end{equation}
and its adjoint operator in $L^2_\sigma (\bbfR^3)$ given by the formula
\begin{equation}\label{L*:eq}
  \mathcal{L^*}z=-\Delta z +(\nabla v_c)^T  z-(v_c \cdot \nabla )z.
\end{equation}
In the following, we study these operators via the
corresponding sesquilinear forms which defined for all $z,v\in H^1_\sigma (\bbfR^3)$ as follows
\begin{equation}\label{bilin:L}
  a_{\mathcal{L}}(z,v)=\int_{\bbfR^3}\nabla z\cdot \nabla v \ud x +\int_{\bbfR^3} (z\cdot \nabla )v_c\cdot v \ud x +\int_{\bbfR^3}(v_c\cdot \nabla)z\cdot v \ud x
\end{equation}
and
\begin{equation}\label{bilin:L*}
  a_{\mathcal{L^*}}(z,v)=\int_{\bbfR^3}\nabla z\cdot \nabla v \ud x +\int_{\bbfR^3} (\nabla v_c)^T z \cdot v \ud x -\int_{\bbfR^3}(v_c \cdot \nabla )z\cdot v \ud x.
\end{equation}

Our goal is to show that both operators $-\mathcal{L}$ and $-\mathcal{L^\ast }$ 
(in fact, their closures in $L^2_\sigma(\bbfR^3)$)
are infinitesimal generators of analytic semigroups of linear operators on $L^2_\sigma(\bbfR^3)$,
provided condition \rf{c0} is satisfied. Here, we use the following abstract criterion.

\begin{proposition}\label{DH}
Let $\mathcal{H}$ be a Hilbert space and let $\mathcal{V}\subset \mathcal{H}$ be a dense subspace. Assume that $\mathcal{V}$ is a Hilbert space with the inner product $(\cdot , \cdot )_{\mathcal{V}}$ and with the norm $\| \cdot \|_{\mathcal{V}}$ such that for a constant $C>0$ we have $ \|x\|_{\mathcal{H}}\leqslant C\|x\|_{\mathcal{V}}$ for all $x\in \mathcal{V}$. Let $a(x,y)$ be a bounded sesquilinear form on $\mathcal{V}$, which defines an operator $A: \mathcal{D}(A)\rightarrow \mathcal{H}$ as follows
\begin{equation*}
   \mathcal{D}(A)=\{ z\in \mathcal{V} :|a(z,v)|\leqslant C \| v\|_{\mathcal{H}}, v\in \mathcal{V}\}, \quad (Az,v)_{\mathcal{H}}=a(z,v).
\end{equation*}
Suppose that for some $\alpha >0$ and $\lambda_0 \in \bbfR$ we have 
\begin{equation}\label{a:ineq}
  \alpha \| z\|_{\mathcal{V}}^2\leqslant \text{\rm Re} \ a(z,z)+\lambda_0\|z\|_{\mathcal{H}}^2.
\end{equation}
Then $-A$ is the infinitesimal generator of a strongly continuous semigroup of linear operators on $\mathcal{H}$ which is holomorphic in a sector $S_\ve=\{s\in \bbfC: |\text{\rm Arg}\ s|<\ve\}$ for some $\ve >0$.
\end{proposition}
The result stated in Proposition \ref{DH} is essentially due to Lions \cite{Lions}. Its proof is a combination of theorems from \cite{Lions} and \cite{P} and we do not include it here, because this is more or less standard reasoning. A detailed proof can be found {\it e.g.}  either in \cite[Prop. 1.1]{DH} or in  \cite[Prop. 1.51]{Ou}.
 
Now, we apply Proposition \ref{DH} to study operator $\mathcal{L}$ and $\mathcal{L^*}$. 
\begin{theorem}\label{th:L}
Assume that $|c|\geqslant c_0$, where $c_0$ is defined in \rf{c0}. Then 
 the operators $\mathcal{-L}$ and $\mathcal{-L^*}$ defined in \rf{L:eq} and \rf{L*:eq} are infinitesimal generators of strongly continuous semigroups of linear operators on $L^2_\sigma (\bbfR^3)$ which are holomorphic in a sector  $\{s\in \bbfC:|\text{\rm Arg }\ s|<\ve\}$ for a certain $\ve=\ve(c) >0$. 
\end{theorem}

\begin{proof}
We apply  Proposition \ref{DH} with  $\mathcal{H}=L^2_\sigma (\bbfR^3)$ and $\mathcal{V}=H^1_\sigma (\bbfR^3)$.
To show that the  sesquilinear forms $a_{\mathcal{L}}$ and $a_{\mathcal{L}^*}$ are bounded on $\mathcal{V}$, it suffices to follow estimates from
\eqref{vc:est:1} and \eqref{vc:est:2}.

Condition \eqref{a:ineq} for  the sesquilinear form $a_{\mathcal{L}}$  defined in \rf{bilin:L}  results immediately 
 the following inequality 
\begin{equation}\label{L:inq}
   \alpha \| \nabla \otimes z\|_2^2 \leqslant  a_{\mathcal{L}}(z,z)
\end{equation}
for a certain $\alpha >0$ and all $z\in H^1_\sigma (\bbfR^3)$.  Here, we would like to recall  
that $\int_{\bbfR^3} (v_c \cdot \nabla)z \cdot z \ud x=0$ for ${\rm div}\  v_c=0$. Hence, estimate \rf{L:inq} is 
a consequence of Hardy--type ine\-qua\-lity~\rf{eq:H}:
\begin{align}\label{L:inq2}
   a_{\mathcal{L}}(z,z)&= \| \nabla \otimes z\|_2^2 +\int_{\bbfR^3} (z\cdot \nabla )v_c \cdot z \ud x \geqslant \big( 1-K(c)\big)\| \nabla \otimes z\|_2^2,
\end{align}
where  $K(c)<1$ for $|c|\geqslant c_0>1$ by \rf{c0}. 
Using Proposition \ref{DH} we complete the proof that the operator $\mathcal{-L}$ generates a holomorphic semigroup of linear operators on $L^2_\sigma (\bbfR^3)$. 

An analogous argument applies to the adjoint operator $\mathcal{-L^*}$, where by   Lemma \ref{l:K jk}, we get
\begin{align*}
   a_{\mathcal{L^*}}(z,z)&= \| \nabla \otimes z\|_2^2 +\int_{\bbfR^3} (\nabla v_c)^T z \cdot z \ud x \\
   &\geqslant \| \nabla \otimes z\|_2^2 -\int_{\bbfR^3}\sum_{j,k=1}^3|\partial_{x_j}v_c^k| |z_j| |z_k| \geqslant  \big( 1-K(c)\big)\| \nabla \otimes z\|_2^2 .
\end{align*}
Applying Proposition \ref{DH}, we complete the proof of Theorem \ref{th:L}.
\end{proof} 

The following corollaries describe typical properties of generators of analytic semigroups.
We state them for the operator $\mathcal{L}$, however, they are obviously valid for the adjoint operator $\mathcal{L}^*$, as well.

\begin{corollary}\label{l:sqrt L}
Under the assumptions of Theorem \ref{th:L}, the following inequality 
\begin{equation}\label{sqrt L}
   \| \nabla \otimes z\|_2\leqslant \big(1-K(c)\big)\|\mathcal{L}^{\nicefrac{1}{2} }z\|_2 
\end{equation}
holds true for all $z\in \dot H^1_\sigma (\bbfR^3)$.
\end{corollary}

\begin{proof}
By the definition of a square root of nonnegative operators, we have $\| \mathcal{L}^{\nicefrac{1}{2} }z\|_2^2=a_{\mathcal{L}}(z,z)$. Hence to complete this proof, it suffices to recall inequality \rf{L:inq2}.
\end{proof}

\begin{corollary}\label{c:w0 norm}
Under the assumptions of Theorem \ref{th:L},
\begin{equation}\label{w0 norm}
   \| \se^{-t\mathcal{L}}z_0\|_2\leqslant \|z_0\|_2
\end{equation}
for all $z_0\in L^2_\sigma (\bbfR^3)$ and $t>0$.
\end{corollary}

\begin{proof}
Multiplying equation \rf{z:eq} by $z$ and integrating over $\bbfR^3$, we easily obtain energy equality
\begin{equation*}
  \frac{1}{2}\frac{\ud}{\ud t}\|z(t)\|^2_2+\| \nabla \otimes z(t)\|_2^2 +\int_{\bbfR^3}(z\cdot \nabla )v_c\cdot z \ud x=0,
\end{equation*}
because $\int_{\bbfR^3}(v_c \nabla)z\cdot z \ud x=0$ by the condition ${\rm div}\ v_c=0$. Hence, the Hardy-type inequality \rf{eq:H} yields
\begin{equation}\label{l en inq} 
   \frac{1}{2}\frac{\ud}{\ud t}\|z(t)\|^2_2 + \big( 1-K(c) \big)\| \nabla \otimes z(t)\|_2^2\leqslant 0 \ ,
\end{equation}
where $1-K(c)\geqslant 0$ by \rf{c0}. Now, it is sufficient to integrate from $0$ to $t$ to obtain the inequality \rf{w0 norm}.    
\end{proof}

\begin{corollary}\label{sg prop:c}
There exists a constant $C>0$ such that for all $z_0\in L^2_\sigma (\bbfR^3)$ the following inequalities
\begin{equation}\label{sg prop}
   \| \mathcal{L}\se^{-t\mathcal{L}}z_0\|_2\leqslant Ct^{-1}\|z_0\|_2
\end{equation}
and
\begin{equation}\label{sg prop2}
   \| \mathcal{L}^{\nicefrac{1}{2}}\se^{-t\mathcal{L}}z_0\|_2\leqslant Ct^{-\frac{1}{2}}\|z_0\|_2
\end{equation}
hold true for all $t>0$.   
\end{corollary}

\begin{proof}
Inequality \rf{sg prop} is the well-known property of analytic semigroups of linear operators (see {\it e.g. }\cite[Theorem 5.2]{P} for more details). Using properties of a square root of a nonnegative operator, the Schwarz inequality, inequality \rf{sg prop} and Corollary \ref{c:w0 norm}, we obtain 
\begin{equation*}
   \| \mathcal{L}^{\frac{1}{2}} \se^{-t\mathcal{L}}\psi \|_2^2= |(\mathcal{L}\se^{-t\mathcal{L}}\psi,\se^{-t\mathcal{L}}\psi)|\leqslant \|\mathcal{L}\se^{-t\mathcal{L}}\psi\|_2\|\se^{-t\mathcal{L}}\psi \|_2\leqslant Ct^{-1}\| \psi \|_2^2 \ 
\end{equation*}
for all $t>0$.
\end{proof}

\begin{corollary}\label{c:lim L2} 
Under the assumptions of Theorem \ref{th:L}, for all $z_0 \in L^2_\sigma (\bbfR^3 )$ 
\begin{equation}\label{eq:lim L2}
   \lim_{t\rightarrow \infty } \| \se^{-t\mathcal{L}}z_0\|_2 =0.
\end{equation}
\end{corollary}

\begin{proof}
Let $z_0 \in L^2_\sigma (\bbfR^3)$. Since the range of the operator $\mathcal{L}$ is a dense subspace of $L^2_\sigma (\bbfR^3)$, for every $\ve>0$ there exists a function $\varphi \in {\rm Range}(\mathcal{L})$ such that $\| \varphi -z_0\|_2<\ve $. Consequently, applying Corollary \ref{sg prop:c} and Corollary \ref{c:w0 norm}, we obtain
\begin{align*}
   \| \se^{-t\mathcal{L}}z_0\|_2&\leqslant \| \se^{-t\mathcal{L}}(z_0 - \varphi )\|_2+\| \se^{-t\mathcal{L}}\varphi \|_2\leqslant \ve + \| \mathcal{L}\se^{-t\mathcal{L}}\psi \|_2 \leqslant \ve + Ct^{-1}\| \psi\|_2 ,
\end{align*}
where $\psi \in \mathcal{D}(\mathcal{L})$. Hence, $\limsup_{t\rightarrow \infty }\| \se^{-t\mathcal{L}}z_0\|_2 \leqslant \ve$. Since $\ve>0$ is arbitrarily small, we complete the proof.
\end{proof}

\begin{corollary}\label{c:lim int et}
Under the assumptions of Theorem \ref{th:L}, for all $z_0 \in L^2_\sigma (\bbfR^3 )$
\begin{equation}\label{eq:lim int et}
   \lim_{t\rightarrow \infty }\frac{1}{t}\int_0^t \| \se^{-s\mathcal {L} }z_0\|_2 \ud s =0\, .
\end{equation}
\end{corollary}

\begin{proof}
Substituting $s=t\tau $, we get
\begin{equation*}
   \frac{1}{t}\int_0^t \|\se^{-s\mathcal{L}}z_0\|_2=\int_0^1 \|\se^{-t\tau \mathcal{L} }z_0\|_2 \ud \tau \,.
\end{equation*}
Now, the desired result follows from Corollaries \ref{c:w0 norm} and \ref{c:lim L2} combined with the Lebesgue dominated convergence theorem.
\end{proof}

We conclude this section by showing the decay estimates of the semigroup $\se^{-t\mathcal{L}}$.

\begin{proposition}[Hypercontractivity]\label{th:hiper} 
Assume that $|c|\geqslant c_0>1$, where $c_0$ satisfies \rf{c0}. For each $p\in (\frac{6}{5},2)$ there exists a constant $C=C(p)>0$ such that for every $z_0 \in L^p_\sigma (\bbfR^3)$
\begin{equation}\label{semi:eq}
   \| \se^{-t\mathcal{L}}z_0\|_2\leqslant Ct^{-\frac{3}{2}(\frac{1}{p}-\frac{1}{2})}\|z_0\|_p 
\end{equation} 
for all $t>0$.
\end{proposition}

\begin{proof}
First, we consider the semigroup generated by the adjoint operator $\mathcal{L}^\ast $ defined in \rf{L*:eq}. For every $q\in (2,6)$, using the Gagliardo-Nirenberg-Sobolev inequality, we obtain
\begin{align*}
   \| \se^{-t\mathcal{L}^\ast} z_0\|_q^q \leqslant C \|\se ^{-t\mathcal{L}^\ast} z_0\|_2^{\frac{1}{2}(6-q)}\|\nabla \se^{-t\mathcal{L}^\ast} z_0\|_2^{\frac{3}{2}(q-2)}.
\end{align*}
Next, applying Corollaries \ref{c:w0 norm} and \ref{sg prop:c} with $\mathcal{L}$ replaced by $\mathcal{L^*}$, we get
\begin{align*}
   \| \se^{-t\mathcal{L}^\ast} z_0\|_q^q &\leqslant C\|z_0\|_2^{\frac{1}{2}(6-q)}\|{(\mathcal{L}^\ast )}^{\nicefrac{1}{2}} \se^{-t\mathcal{L}^\ast} z_0\|_2^{\frac{3}{2}(q-2)}\\
   &\leqslant C \Big(t^{-\frac{3}{2}(\frac{1}{2}-\frac{1}{q})}\|z_0\|_2\Big)^q \qquad \textrm{for all}\ \ t>0.
\end{align*}
Hence, by a duality argument, we immediately deduce the inequality 
\begin{equation*}
  \|\se^{-t\mathcal{L}}z_0\|_2\leqslant Ct^{-\frac{3}{2}(\frac{1}{p}-\frac{1}{2})}\|z_0\|_p \quad \textrm{for all}\ \ t>0,
\end{equation*}
with $p=\frac{q}{q-1}\in (\frac{6}{5},2)$.
\end{proof}
\section{Asymptotic stability of weak solutions}
\setcounter{equation}{0}

To show the decay of $\|w(t)\|_2$, we use the approach from \cite{BM3} which involves the weak $L^p$-spaces. By this reason, let us recall the weak Marcinkiewicz $L^p$-spaces ($1<p<\infty $), denoted as usual by $L^{p,\infty }=L^{p,\infty }(\bbfR)$, which belong to the scale of the Lorentz spaces and contain measurable functions $f=f(x)$ satisfying the condition
\begin{equation}\label{Lsp}
   |\{x\in \bbfR : |f(x)|>\lambda \}|\leqslant C\lambda ^{-p}
\end{equation}
for all $\lambda >0$ and a constant $C$. One check that \rf{Lsp} is equivalent to
\begin{equation*}
   \int_{E}|f(x)| \ud x \leqslant \widetilde{C}|E|^{\frac{1}{q}}
\end{equation*}
for every measurable set $E$ with a finite measure, another constant $\widetilde{C}$, and $\frac{1}{q}+\frac{1}{p}=1$. This fact allows us to define the norm in $L^{p,\infty}$
\begin{equation}\label{norm:Lpw}
   \|f\|_{p,\infty }=\sup \Big\{ |E|^{-1+\frac{1}{q}}\int_{E}|f(x)| \ud x : E\in \mathcal{B} \Big\}
\end{equation}
where $\mathcal{B}$ is the collection of all Borel sets with a finite and positive measure. Recall the well-known imbedding $L^p\subset L^{p,\infty }$ being the consequence of the Markov inequality $|\{ x\in \bbfR : |f(x)|>\lambda \}|\leqslant \lambda^{-p}\int_{\bbfR}|f(x)|^p \ud x $. Moreover, the following inequalities hold true: \text{\it the weak H\"older inequality}:
\begin{equation}\label{wH:inq}
   \|fg\|_{r,\infty }\leqslant \|f\|_{p,\infty }\|g\|_{q,\infty }
\end{equation}
for every $1<p\leqslant \infty $ (here $L^{\infty , \infty }=L^\infty $), $1<q<\infty $ and $1<r<\infty $ satisfying $\frac{1}{r}=\frac{1}{q}+\frac{1}{p}$, and \text{\it the weak Young inequality }
\begin{equation}\label{wY:inq}
    \|f\ast g\|_{r,\infty }\leqslant C\|f\|_{p,\infty }\|g\|_{q,\infty }
\end{equation}
for every $1<p<\infty $, $1<q<\infty $ and $1<r<\infty $ satisfying $1+\frac{1}{r}=\frac{1}{p}+\frac{1}{q}$. We refer reader to \cite{BM3} for the proofs of the results stated above. 

The following lemma is extracted from reasonings contained in \cite{BM3} and its proof is based on properties of the weak $L^p$-spaces.

\begin{lemma}\label{l:weak Y}
Assume that $f\in L^1 \big( (0,+\infty )\big)$. For every $\alpha \in (1, +\infty ]$ there exists a constant $C>0$ such that 
\begin{equation}\label{inq:weak Y}
     \frac{1}{t} \int_0^t \Big(|\cdot |^{-\frac{1}{2}} \ast \big(|\cdot |^{-\frac{1}{\alpha}}f^{\frac{3}{4}}\big) \Big) (s) \ud s\leqslant Ct^{-\frac{1}{4}-\frac{1}{\alpha }}\|f\|_1^{\frac{3}{4}}.
\end{equation}
for all $t>0$. Here, for $\alpha =+\infty $, the quantity $\nicefrac{1}{\alpha }$ should be replaced by $0$.
\end{lemma}

\begin{proof}
First, consider $\nicefrac{1}{\alpha }=0$.
Using the definition of the norm \rf{norm:Lpw} in the weak $L^p$-spaces, we have 
\begin{equation*}
   L_1=\frac{1}{t}\int_0^t \Big( |\cdot |^{-\frac{1}{2}} \ast f^{\frac{3}{4}}\Big) (s) \ud s\leqslant t^{-\frac{1}{q}}\big\| |\cdot |^{-\frac{1}{2}}\ast f^{\frac{3}{4}}\big\|_{q,\infty }
\end{equation*}
for every $q\in (1, \infty )$ to be chosen later. Since, $\|g\|_{p,\infty }\leqslant C\|g\|_p$ for all $p\in [1,\infty ]$, the weak Young inequality \rf{wY:inq} implies
\begin{equation*}
  L_1\leqslant Ct^{-\frac{1}{q}}\big\| |\cdot |^{-\frac{1}{2}} \big\|_{2,\infty } \|f\|_{\frac{3}{4}r}^{\frac{3}{4}},
\end{equation*}  
where $1+\frac{1}{q}=\frac{1}{2}+\frac{1}{r}$. Hence, however, we require $\nicefrac{3r}{4}=1$, hence $q=4$. Since the function $|\cdot |^{-\frac{1}{2}}\in L^{2,\infty }\big( (0,+\infty ) \big)$, we complete the proof of \rf{inq:weak Y} in case $\alpha =+\infty $.

For $\alpha \in (0, +\infty )$, applying an analogous argument involving the definition of the norm \rf{norm:Lpw} in the weak $L^p$-spaces, the weak Young inequality \rf{wY:inq}, we obtain
\begin{align*}
   L_2&=\frac{1}{t} \int_0^t \Big(|\cdot |^{-\frac{1}{2}} \ast |\cdot |^{-\frac{1}{\alpha}}f^{\frac{3}{4}} \Big) (s) \ud s\leqslant t^{-\frac{1}{p}}\big\| |\cdot |^{-\frac{1}{2}}\ast |\cdot |^{-\frac{1}{\alpha }}f^{\frac{3}{4}}\big\|_{p,\infty }\\
   & \leqslant t^{-\frac{1}{p}}\big\| |\cdot |^{-\frac{1}{2}} \big\|_{2,\infty } \big\| |\cdot |^{-\frac{1}{\alpha}}f^{\frac{3}{4}}\big\|_{q,\infty }
\end{align*}
for every $p,q \in (1, \infty )$ satisfying $1+\frac{1}{p}=\frac{1}{2}+\frac{1}{q}$.
Now, the weak H\"older inequality \rf{wH:inq} gives us 
\begin{equation*}
    L_2\leqslant Ct^{-\frac{1}{p}}\big\| |\cdot |^{-\frac{1}{2}} \big\|_{2,\infty }\big\| |\cdot |^{-\frac{1}{\alpha }} \big\|_{\alpha ,\infty }\| f\|_{\frac{3}{4}r}^{\frac{3}{4}},
\end{equation*}
where $\frac{1}{q}=\frac{1}{r}+\frac{1}{\alpha }$. Assuming $\nicefrac{3r}{4}=1$, we get $\nicefrac{1}{p}=\nicefrac{1}{4}+\nicefrac{1}{\alpha }$. 
\end{proof}

\begin{lemma}\label{l:P}
There exists $C>0$ such that for all $v,w \in H^1_\sigma (\bbfR^3)$ and   $\psi \in L^2 _\sigma (\bbfR^3)$ the following estimate 
\begin{equation}\label{P:inq}
\big((w\cdot \nabla )v,  e^{-t\mathcal{L}^*} \psi \big)
  \leqslant Ct^{-\frac{1}{2}}(\|w\|_2\|v\|_2)^{\frac{1}{4}}(\| \nabla w\|_2 \| \nabla v\|_2)^{\frac{3}{4}}\|\psi\|_2
\end{equation}
holds true for all $t>0$.
\end{lemma}

\begin{proof} 
By inequalities \rf{sqrt L} and  \rf{sg prop2}
(with $\mathcal{L}$ replaced by $\mathcal{L}^*$), 
we have $\| \nabla \se^{-t\mathcal{L}^*}\psi \|_2\leqslant C t^{-\frac{1}{2}}\| \psi \|_2$. Hence, a direct calculation involving the integration by parts, the H\"older inequality and 
inequality \eqref{w0 norm}
leads to
\begin{align*}
  |\big((w\cdot \nabla )v,  e^{-t\mathcal{L}^*} \psi \big)|&=|\big(v,w\cdot\nabla \se^{-t\mathcal{L}^*}\psi\big )|\leqslant \|v\|_4\|w\cdot \nabla\se^{-t\mathcal{L}^*}\psi\|_{\frac{4}{3}}\\
&\leqslant\|v\|_4\|w\|_4\|\nabla \se^{-t\mathcal{L}}\psi\|_2 \leqslant Ct^{-\frac{1}{2}}\|v\|_4\|w\|_4\|\psi\|_2.
\end{align*}
Hence, the proof is completed by 
 the Sobolev inequality
$
 \|v\|_4 \leqslant C\| \nabla v\|_2^{\frac{3}{4}}\|v\|_2^{\frac{1}{4}},
$
which holds true for  all  $v\in H^1_\sigma (\bbfR^3)$.
\end{proof}

\begin{proof}[Proof of Theorem \ref{lim w:th}]
Let $w$ be a weak solution of system \rf{w:eq}--\rf{w:ini} in the space $X_T$ defined in Theorem \ref{th:ex} which satisfies the strong energy inequality \rf{en:inq}. First, we show that
\begin{equation}\label{int:lim}
  \lim_{t\rightarrow \infty} \frac{1}{t}\int_0^t \|w(s)\|_2\ud s=0.
\end{equation}
Observe that inequality \rf{en:inq} implies $\|w(\cdot )\|_2 \in L^{\infty }(0, \infty )$ and $\| \nabla w(\cdot )\|_2^2 \in L^1(0,\infty )$. 
Now,  with an arbitrary $\psi\in C^\infty_{c, \sigma}(\bbfR^3)$,
we substitute $\varphi(\tau) =\se^{-(s-\tau)\mathcal{L}^*} \psi$ into equation \eqref{weak}
(with $s=0$ and $t$ replaced by $s$)
to obtain
the following integral formulation of  problem \rf{w:eq}-\rf{w:ini}
\begin{equation}\label{duhamel}
    \big(w(s), \psi\big)=
\big(\se^{-s\mathcal{L}}w_0, \psi\big)
+\int_0^s 
\big( (w\cdot \nabla w)(\tau),
\se^{-(s-\tau)\mathcal{L}^*}\psi\big)
\ud \tau .
\end{equation}
Here, in calculations  leading  to \eqref{duhamel},
 one should transform the last term on the right-hand side of \rf{weak} in the following way
\begin{equation*}
\begin{split}
\int_0^s \big(w, \varphi_\tau\big)\,d\tau
&=\int_0^s \big(w, \mathcal{L}^* \varphi\big)\,d\tau=\int_0^s \big(\mathcal{L} w,  \varphi\big)\,d\tau\\
&= \int_0^s
\Big[ 
\big( \nabla w, \nabla \varphi\big)
+ \big(w\cdot  \nabla v_c , \varphi\big)
+\big(v_c \cdot \nabla w, \varphi\big) 
 \Big]\,d\tau,
\end{split}
\end{equation*}
because  ${\rm div}\, \varphi=0$.

Hence, applying Lemma \ref{l:P} to estimate the nonlinear term  in \eqref{duhamel} and the $L^2$-duality argument, we get
\begin{align*}
     \|w(s)\|_2 
     &\leqslant \| \se^{-s \mathcal{L}}w_0\|_2+ C\int_{0}^{s} (s-\tau)^{-\frac{1}{2}}\|w(\tau )\|_2^{\frac{1}{2}}\| \nabla w(\tau )\|_2^{\frac{3}{2}} \ud \tau \\
     &\leqslant \| \se^{-s \mathcal{L}}w_0\|_2+ C\sup_{\tau >0}\|w(\tau )\|_2^{\frac{1}{2}}\int_{0}^{s} (s-\tau)^{-\frac{1}{2}}\| \nabla w(\tau )\|_2^{\frac{3}{2}} \ud \tau 
\end{align*} 
since $\|w(\cdot )\|_2 \in L^{\infty }(0,\infty )$. Integrating from $0$ to $t$ and multiplying by $\nicefrac{1}{t}$, we obtain
\begin{equation*}
   \frac{1}{t}\int_{0}^{t} \|w(s)\|_2 \ud s \leqslant \frac{1}{t}\int_{0}^{t}\| \se^{-s \mathcal{L}}w_0\|_2\ud s + \frac{C}{t}\int_0^t \Big( |\cdot |^{-\frac   {1}{2}}\ast \| \nabla w(\cdot )\|_2^{\frac{3}{2}}\Big) (s) \ud s\, .
\end{equation*}
Now, since $\| \nabla w\|_2^2 \in L^1\big( (0,+\infty )\big)$, we apply Lemma \ref{l:weak Y}  with $\nicefrac{1}{\alpha} =0$ 
to get the estimate
\begin{align}\label{int:lim1}
  \frac{1}{t}\int_{0}^{t}\| w(s)\|_2 \ud s \leqslant \frac{1}{t}\int_0^t\| \se^{-s \mathcal{L}}w_0\|_2 \ud s + Ct^{-\frac{1}{4}},
\end{align} 
which proves \rf{int:lim} by Corollary \ref{c:lim int et}.

Next, notice that, by the strong energy inequality \rf{en:inq}, $\| w(t)\|_2$ is a non-increasing function of $t$ for almost all $t\geqslant 0$. Hence, for  $t>0$ we obtain
\begin{align}\label{int:lim3}
   \| w(t)\|_2= \frac{1}{t} \| w(t)\|_2 \int_0^t \ud s\leqslant \frac{1}{t}\int_0^t \| w(s)\|_2 \ud s \, .
\end{align}
The proof is completed by \rf{int:lim}.
\end{proof}

\begin{proof}[Proof of Corollary \ref{cor. decay}]
Using the decay estimate from Proposition \ref{th:hiper}, we have 
\begin{equation*}
    \frac{1}{t} \int_0^t \| \se^{-s\mathcal{L}}w_0\|_2 \ud s \leqslant Ct^{-\frac{3}{2}(\frac{1}{p}-\frac{1}{2})}\|w_0\|_p
\end{equation*}
for each $p\in (\frac{6}{5}, 2)$ and all $t>0$. Applying this inequality in \rf{int:lim1} and recalling \rf{int:lim3}, we complete the proof of the corollary in the case of $p\in [\frac{3}{2}, 2)$. 

Now, we notice that for $p\in (\frac{6}{5}, \frac{3}{2})$ inequality \rf{int:lim1} implies $\| w(t)\|_2 \leqslant Ct^{-\frac{1}{4}}$ for all $t>0$. Hence, repeating the reasoning from the proof of Theorem \ref{lim w:th} and applying Lemma \ref{l:weak Y}
with $\alpha=8$, we get the estimate 
\begin{align*}
   \frac{1}{t}\int_{0}^{t} \|w(s)\|_2 \ud s &\leqslant \frac{1}{t}\int_{0}^{t}\| \se^{-s \mathcal{L}}w_0\|_2\ud s + \frac{C}{t}\int_0^t \Big( |\cdot |^{-\frac   {1}{2}}\ast | \cdot |^{-\frac{1}{8}}\| \nabla w(\cdot )\|_2^{\frac{3}{2}}\Big) (s) \ud s \\
   &\leqslant Ct^{-\frac{3}{2}(\frac{1}{p}-\frac{1}{2})}\|w_0\|_p +Ct^{-\frac{3}{8}},
\end{align*}
which proves decay estimate \rf{decay} for $p\in [\frac{4}{3},\frac{3}{2})$. Repeating this procedure finitely many times, we complete the proof for each $p\in (\frac{6}{5}, 2)$.
\end{proof}

\end{document}